\documentclass[12pt]{amsart}
\usepackage{amsfonts}
\usepackage{amscd}
\usepackage{amssymb}
\usepackage{graphics}

\usepackage{amsmath}

\usepackage{hyperref}
\hypersetup{colorlinks,citecolor=blue}

\newcommand{\BN}{{\mathbb{N}}}
\newcommand{\BR}{{\mathbb{R}}}
\newcommand{\BC}{{\mathbb{C}}}

\newcommand{\BG}{{\mathbb{G}}}

\newcommand{\gD}{\Delta}
\newcommand{\gd}{\delta}
\newcommand{\gb}{\beta}

\newcommand{\gC}{\Gamma}
\newcommand{\gc}{\gamma}

\newcommand{\gep}{\epsilon}
\newcommand{\gl}{\lambda}
\newcommand{\ga}{\alpha}

\newcommand{\ti}[1]{\tilde{#1}}

\newcommand{\diam}{\text{diam}}

\newcommand{\GL}{\text{GL}}

\newtheorem{prop}{Proposition}[section]
\newtheorem{thm}[prop]{Theorem}
\newtheorem{lem}[prop]{Lemma}
\newtheorem{cor}[prop]{Corollary}

\theoremstyle{definition}

\newtheorem{defn}[prop]{Definition}

\newtheorem{rems}[prop]{Remarks}
\newtheorem{exam}[prop]{Example}

\catcode`\@=11
\long\def\@savemarbox#1#2{\global\setbox#1\vtop{\hsize\marginparwidth 
  \@parboxrestore\tiny\raggedright #2}}
\catcode`\@=12

\begin{document}
\author{Tsachik Gelander}

\date{\today}


\title{Limits of finite homogeneous metric spaces}

\maketitle


 

\section{introduction}
Itai Benjamini asked me if the spheres $S^2$, and in general which manifolds, can be approximated by finite homogeneous metric spaces. \footnote{The motivation came from the recent paper \cite{Itai} which studies limits of scaled transitive graphs.}
 
We will say that a complete metric space is approximable (or can be approximated) by finite homogeneous metric spaces if it is a limit of such in the Gromov--Hausdorff topology. 

\begin{thm}\label{thm}
A metric space $X$ is approximable by finite homogeneous metric spaces if and only if it admits a compact group of isometries $G$ which acts transitively and whose identity connected component $G^\circ$ is abelian.  
\end{thm}

\begin{cor}\label{cor}
If $X$ is approximable by finite homogeneous metric spaces then
$X$ is compact, the connected components of $X$ are inverse limits of tori, and the quotient space of connected components $X/\sim$ is a transitive totally disconnected space hence is either finite or homeomorphic to the Cantor set.
\end{cor}  

We will see that a connected component of $X$ is an inverse limit of tori in the strong, group theoretic, manner: it is homeomorphic to $\varprojlim T_n$ where the $T_n$ are compact abelian Lie groups, and the associated maps are surjective homomorphisms. 

In particular,

\begin{cor}
The only manifolds that can be approximated by finite homogenous metric spaces are tori.
\end{cor} 

The following example illustrates another type of limits of finite homogeneous metric spaces. 

\begin{exam}[The solenoid]
For every $n$, let $T_n$ be a copy of the circle group $\{z\in\BC:|z|=1\}$, and whenever $m$ divides $n$ let $f_{n,m}:T_n\to T_m$ be the $n/m$ sheeted cover $f_{n,m}(z)=z^{n/m}$. Let $G=\varprojlim T_n$ be the inverse limit group, and equip $G$ with any compatible invariant metric. Then $G$ is approximable by finite homogeneous metric spaces, but it is not arcwise connected, hence cannot be approximated by finite transitive graphs (see Remark \ref{rem}.4).  
\end{exam}

\begin{rems}\label{rem}
1. A classical Theorem of Jordan \cite{J} states that for every $n\in\BN$ there is $m\in\BN$ such that any finite subgroup of $\GL_n(\BC)$ admits an abelian subgroup of index $\le m$. Based on this theorem, Turing \cite{Turing} proved  that a connected compact Lie group $G$ can be approximated by finite groups, or in our terminology below is quasi finite, if and only if it is abelian (see also \cite{Kazhdan}\footnote{Kazhdan, who was not aware of \cite{Turing}, reproved Turing's theorem, and in fact a stronger result, dealing with amenable instead of finite groups, and a unitary image which is not necessarily finite dimensional.} and \cite{AGG}). Our result can be considered as a metric version of Jordan's and Turing's theorems.

\medskip
\noindent

2. It follows that for $n\ge 2$, the $n$ dimensional sphere $S^n$ cannot be approximated by finite homogeneous metric spaces. It makes sense to ask how far is the (vertex space of the) soccer ball, with the induced metric from $S^2$, or any other semi-Platonic polygon, from the optimal approximation. It may also be interesting to investigate the asymptotic of the distance of $S^n$ from the set of finite homogeneous spaces.

\medskip
\noindent
3. Note that every homogeneous metric on a finite, as well as infinite, dimensional torus can be approximated by finite homogeneous metric spaces. However, only very specific metrics can be approximated by scaled finite transitive graphs. It might be interesting to classify these metrics.

\medskip
\noindent
4. If an inverse limit of tori $T=\varprojlim T_\ga$ admits a compatible metric that makes it a length space then for every $\ga$ there is $k(\ga)\in\BN$ such that for any $\gb>\ga$ the pre-image in $T_\gb$ of a point in $T_\ga$ has at most $k(\ga)$ connected components. This implies that $T$ splits as a product $\overline{T}_\ga\times T'$ where $\overline{T}_\ga$ is some finite cover of $T_\ga$. It follows that $T$ is actually a torus, i.e. homeomorphic to a finite or infinite product of circles. 

If $X_n$ are scaled transitive finite graphs with number of vertices tends to infinity, and $(X_n)$ has bounded geometry (defined below), then, by \ref{lem:comp}, $\{X_n\}_{n=1}^\infty$ is relatively compact, and it is easy to see that every limit must be a length space. 
It follows from Theorem \ref{thm} that any limit of $(X_n)$ is a (finite or infinite dimensional) torus.

Imposing additional conditions on $X_n$, it is sometimes possible to deduce more information about the limits. For instance, if the $X_n$ satisfy the growth condition $|X_n|\le \text{Diam}(X_n)^q$ (as in \cite{Itai}\footnote{By \cite{Itai} this growth condition implies bounded geometry.}), where $\text{Diam}(X_n)$ denotes the diameter before scaling, it is possible to show that any limit must be a torus of dimension at most $q$. 

For another example, assume that the girth of $X_n$ is bounded below. Then the girth (i.e. the infimal length of a simple loop) of any limit of $(X_n)$ is also bounded below. This forces the limit to be a circle. Hence it follows that $(X_n)$ converges to $S^1$ and in turn that the entropy of $X_n$ is asymptotically the entropy of the circle, i.e. $\forall\gep>0,~E(\gep,X_n)\to 1/\gep$ (see definition below), and the girth of $X_n$ tends to $2$. 
Thus the geometry of the $(X_n)$ is either unbounded or asymptotically the geometry of the circle.
\end{rems}

Here is another situation where the limit is unique and has to be $S^1$.
Recall that a graph is said to be {\it distance transitive} if the automorphism group acts transitively on pairs of vertices of any given distance.

\begin{cor}\label{cor:distance-transitive}
Let $(X_n)$ be a bounded geometry sequence of distance transitive scaled finite graphs and suppose that the number of vertices $|X_n|\to \infty$. Then $(X_n)$ converges to $S^1$ with the length distance. In particular $\forall\gep>0,~E(\gep,X_n)\to 1/\gep$.
\end{cor}

{\bf Acknowledgment:} I would like to thank Itai Benjamini for stimulating discussions, Udi Hrushovski and Benji Weiss for telling me about Kazhdan's and Turing's papers
and Yair Glasner for suggesting the problems answered in Remark \ref{rem}.4 and Corollary \ref{cor:distance-transitive}. The research was supported by the ISF and the ERC.


\section{Some compactness arguments}\label{sec:cmp}
Let $E:\BR^{+}\to\BN$ be a non-increasing function. We shall say that a metric space $Y$ has $E$ bounded geometry if for every $\gep>0$ the $\gep$-entropy of $Y$ is at most $E(\gep)$, i.e. any subset of $Y$ of size $E(\gep)+1$ admits two points of distance at most $\gep$. We shall denote the $\gep$-entropy of a compact space $Y$ by $E(\gep,Y)$. We will say that a family $\mathcal{A}$ of compact metric spaces has uniformly bounded geometry (or shortly, bounded geometry) if there is a function $E(\gep)$ such that $E(\gep,Y)\le E(\gep)$ for every $Y\in\mathcal{A}$. 
First recall:

\begin{lem}[Gromov's compactness criterion, \cite{BH} Theorem 5.41]\label{lem:comp}
A set $\mathcal{A}$ of compact metric spaces is relatively compact with respect to the Gromov--Hausdorff distance iff the elements in $\mathcal{A}$ have uniformly bounded diameter and uniformly bounded geometry.
\end{lem}

We will usually denote all distances by $d$, omitting references to the spaces in which they are measured. By the distance between subsets of a given metric space we always mean the Hausdorff distance and the distance between two compact metric spaces is always the Gromov--Hausdorff distance.

The isometry group of a compact metric space $Y$ is compact and equipped with the bi-invariant metric $d(g_1,g_2)=\max_{y\in Y}d_Y(g_1\cdot y,g_2\cdot y)$. Note that $\text{diam}(\text{Isom}(Y))\le\diam(Y)$.

\begin{lem}
If $Y$ has $E$-bounded geometry then $\text{Isom}(Y)$ has $E'$-bounded geometry, where $E'(\gep)=E({\gep\over 4})^{E({\gep\over 4})}$.
\end{lem}

\begin{proof}
Let $\mathcal{F}$ be an $\gep\over 4$-net in $Y$ of size $|\mathcal{F}|\le E({\gep\over 4})$. By the Pigeonhole Principle, in any set $\{g_i\}_{i\in I}$ of $|I|=|\mathcal{F}|^{|\mathcal{F}|}+1$ isometries of $Y$ there are $i,j\in I$ such that $i\ne j$ and $d(g_i\cdot f,g_j\cdot f)\le{\gep\over 2}$ for every $f\in\mathcal{F}$.
Since for every $y\in Y$ there is $f\in\mathcal{F}$ with $d(y,f)<{\gep\over 4}$ we deduce that $d(g_i,g_j)=\max_{y\in Y}d(g_i\cdot y,g_j\cdot y)\le \gep$.
\end{proof}

It follows that if $X_n$ is a sequence of compact metric spaces which converges to a limit $X$, then the associated isometry groups $\text{Isom}(X_n)$ have uniformly bounded geometry. Hence for any choice of closed subgroups $G_n\le \text{Isom}(X_n)$ there is a, not necessarily unique, limit. 
We claim that every such limit $G$ is isometric to a subgroup of $\text{Isom}(X)$. To see this let us take a subsequence such that $G_{n_k}\to G$ and for every $n_k$ let us fix 
$$
 \phi_{n_k}:X_{n_k}\to X,~\ti\phi_{n_k}:X\to X_{n_k},~\psi_{n_k}:G_{n_k}\to G~\text{and}~\ti\psi_{n_k}:G\to G_{n_k}
$$ 
such that the pair $(\phi_{n_k},\ti\phi_{n_k})$ (resp. $(\psi_{n_k},\ti\psi_{n_k})$) form an $\gep_{n_k}$-isometric equivalence\footnote{I.e. they are $(1,\gep)$-quasi isometries and their compositions are $\gep$-close to the identities.} between $X_{n_k}$ and $X$ (resp. between $G_{n_k}$ and $G$), and $\gep_{n_k}\to 0$. 
Since bounded geometry behaves nicely with direct products and restrictions to subsets, by taking a further sub-sequence $n_{k_l}$ we get that the spaces of triples 
$$
 Z_{n_{k_l}}=\{ (g,x,g\cdot x):g\in G_{n_{k_l}}, x\in X_{n_{k_l}}\}\subset G_{n_{k_l}}\times X_{n_{k_l}}\times X_{n_{k_l}}
$$
converges to a subset $Z\subset G\times X\times X$, where for products we take the sup metric, and the restrictions of the maps
$$
 ((\psi_{n_{k_l}},\phi_{n_{k_l}},\phi_{n_{k_l}}),(\ti\psi_{n_{k_l}},\ti\phi_{n_{k_l}},\ti\phi_{n_{k_l}})),
 \footnote{Note that the restriction of this pair of maps to $Z_{n_{k_l}}\times Z$ does not necessarily range inside $Z\times Z_{n_{k_l}}$ but this should cause no ambiguity.}
$$
form $\gep_{n_{k_l}}'$-equivalence between $Z_{n_{k_l}}\subset G_{n_{k_l}}\times X_{n_{k_l}}\times X_{n_{k_l}}$ and $Z\subset G\times X\times X$ with $\gep_{n_{k_l}}'\to 0$. 

Since, for every ${n_{k_l}}$, the projection to the product of the first two factors 
$$
 Z_{n_{k_l}}\to G_{n_{k_l}}\times X_{n_{k_l}}
$$ 
is $1$-Lipschitz one-to-one and onto with a $2$-Lipschitz inverse,
we deduce that the projection $Z\to G\times X$ share these properties as well. Denoting by $p^{-1}:G\times X\to Z$ the inverse of this projection and by $\pi$ the projection to the third factor, allows us to define an action of every $g\in G$ on $X$ by: $g\cdot x:=\pi\circ p^{-1}(g,x)$. Since this is a "limit of actions by isometries of groups", i.e. 
$$
 \forall (g,x)\in G\times X,~ g\cdot x=\lim_{n_{k_l}}\phi_{n_{k_l}}(\ti\psi_{n_{k_l}}(g)\cdot\ti\phi_{n_{k_l}}(x)),
$$
it follows that $g\cdot$ is an isometry of $X$ for every $g\in G$ and that $\{g\cdot:g\in G\}$ is a group. Finally, from the way that the metrics on $G_n$ and on $\text{Isom}(X)$ are defined, it follows that the map $g\mapsto g\cdot$ is an isometry between $G$ and its image in $\text{Isom}(X)$.
In addition, if for all $n$, $G_n$ acts transitively on $X_n$ then for every $x,y\in X$ we can choose $g_{n_{k_l}}\in G_{n_{k_l}}$ such that 
$g_{n_{k_l}}\cdot \ti\phi_{n_{k_l}}(x)=\ti\phi_{n_{k_l}}(y)$, and obviously any limit point $g_0$ of $\psi_{n_{k_l}}(g_{n_{k_l}})$ satisfies $g_0\cdot x=y$. Let us summarize the above discussion:

\begin{prop}\label{prop:G}
Suppose that $X_n$ are compact metric spaces, $G_n\le\text{Isom}(X_n)$ are closed subgroups of the corresponding isometry groups and that $X_n\to X$. Then $\{ G_n\}$ is pre-compact and any limit point $G$ of $(G_n)$ is isometric to a subgroup of $\text{Isom}(X)$.\footnote{Note that the limit group $G$ could be strictly smaller than $\text{Isom}(X)$ even for the sequance $G_n=\text{Isom}(X_n)$.} Furthermore, if for all $n$, $G_n$ acts transitively on $X_n$ then every limit $G$ acts transitively on $X$. 
\end{prop}

In order to simplify let us, abusing the above notations, omit the subscript $k_l$ and assume that 
$$
 Z_n:=\{(g,x,g\cdot x):g\in G_n,x\in X_n\}
$$ 
converges, as above, to 
$$
 Z:=\{(g,x,g\cdot x):g\in G,x\in X\}
$$ 
and that 
$$
 (\theta_n,\ti\theta_n):=((\psi_{n},\phi_{n},\phi_{n}),(\ti\psi_{n},\ti\phi_{n},\ti\phi_{n}))
$$
form corresponding $\gep_n$-equivalences with $\gep_n\to 0$, i.e. 
$$
 \theta_n:G_n\times X_n\times X_n\to G\times X\times X~\text{and}~ \ti\theta_n:G\times X\times X\to G_n\times X_n\times X_n
$$
are $\epsilon_n$-isometries whose compositions are $\gep_n$-close to the corresponding identities, and 
$$
 d(\theta_n(Z_n),Z)\le\gep_n~\text{and}~d(\ti\theta_n (Z),Z_n)\le\gep_n.
$$ 

Recall that a map $f:A\to H$ from an abstract group $A$ to a metric group $H$ is called an $\gep$-quasi morphism if 
$$
 \forall a,b\in A,~d(f(a)f(b),f(ab))\le\gep.
$$

\begin{lem}\label{lem:8epsilon}
The map $\psi_n:G_n\to G$ is an $11\gep_n$-quasi morphism. 
\end{lem}

\begin{proof}

Note that for $(h_1,y_1,z_1),(h_2,y_2,z_2)\in Z$ we have $d(z_1,z_2)\le d(h_1,h_2)+d(y_1,y_2)$. This together with the fact that $d(\theta_n(Z_n),Z)\le\gep_n$, implies:

\medskip
\noindent{\bf Claim:}
$\forall \gd\ge0,~g\in G_n,~x\in X$ and $x'\in X_n$ such that $d(\phi_n(x'),x)\le \gd$ we have
$$
 d(\psi_n(g)\cdot x,\phi_n(g\cdot x'))\le 3\gep_n+\gd.
$$

To prove the claim, pick a point $(h,y,z)\in Z$ of distance $\le\gep_n$ from 
$$
 \theta_n (g,x',g\cdot x')=(\psi_n(g),\phi_n(x'),\phi_n(g\cdot x')).
$$ 
Then $d(h,\psi_n(g))\le\gep_n$ and $d(y,x)\le d(y,\phi_n(x'))+d(\phi_n(x'),x)\le \gep_n+\gd$, which implies, since both $(h,y,z)$ and $(\psi_n(g),x,\psi_n(g)\cdot x)$ are in $Z$, that $d(z,\psi_n(g)\cdot x)\le 2\gep_n+\gd$. The claim follows since $d(\phi_n (g\cdot x'),z)\le\gep_n$.

\medskip

Now consider $g,h\in G_n$ and $x\in X$. Picking $x'\in X_n$ such that $d(\phi_n(x'),x)\le\gep_n$ we get from the claim (with $\gd=\gep_n$) that
$$
 d(\psi_n(g)\cdot x,\phi_n(g\cdot x'))\le 4\gep_n~\text{and}~d(\psi_n(hg)\cdot x,\phi_n(hg\cdot x'))\le 4\gep_n
$$
and (taking $\gd=0$)
$$
 d(\psi_n(h)\cdot \phi_n(g\cdot x'),\phi_n(hg\cdot x'))\le 3\gep_n.
$$
Thus
$$
 d(\psi_n(h)\cdot\psi_n(g)\cdot x,\psi_n(hg)\cdot x)\le 11\gep_n,
$$
and since $x$ is arbitrary and
$$
 d(\psi_n(h)\cdot\psi_n(g),\psi_n(hg))=\max_{x\in X}d(\psi_n(h)\cdot\psi_n(g)\cdot x,\psi_n(hg)\cdot x),
$$
the lemma is proved.
\end{proof}


\section{Quasi finite groups}\label{sec:quasi}

\begin{defn}
We will say that a metric group $H$ is {\it quasi finite}\footnote{In \cite{Turing}, Turing used the term ``approximable by finite groups" for the same notion, but here we already used a similar term to express something different.}
if for every $\gep>0$ there is a finite group $A$ and an $\gep$-quasi morphism $f:A\to H$ with an $\gep$-dense image. 
\end{defn}

For example, in the situation of the previous section, if we  
suppose in addition that the spaces $X_n$ are finite, then the groups $G_n\le\text{Isom}(X_n)$ are finite as well, and hence $G$ is quasi finite (c.f. Proposition \ref{prop:G} and Lemma \ref{lem:8epsilon}).

The aim of this section is to give a classification of quasi finite compact groups (see Proposition \ref{prop:turing}). 
Note that being quasi finite is a topological property, i.e. independent of the chosen metric. Thus when studying this notion for compact groups we may restrict ourselves to bi-invariant metrics only. 

Let $H$ be a compact group with a bi-invariant metric. If $N\lhd H$ is a closed normal subgroup, the induced metric on the quotient $H/N$ is defined by the obvious way. Clearly if $H$ is quasi finite then so is $H/N$.


\begin{lem}\label{lem:open}
If $H$ is quasi finite and $O\le H$ is an open subgroup then $O$ is also quasi finite.
\end{lem}

\begin{proof}
Let $\gd>0$ be smaller than the distance from $O$ to the nearest non-trivial coset $hO$. It is easy to see that if $f:F\to H$ is a $\gd$-quasi morphism then $f^{-1}(O)$ is a subgroup of $F$. Additionally, if $f(F)$ is $\gd$-dense in $H$ then $f(f^{-1}(O))=f(F)\cap O$ is $\gd$-dense in $O$.
\end{proof}

\begin{lem}\label{lem:H^0}
Suppose that a compact group $H$ is quasi finite. Then the identity connected component $H^\circ$ is also quasi finite.
\end{lem} 

\begin{proof}
Let $\gep>0$. Since $H/H^\circ$ is totally disconnected, the $\gep$-neighborhood of the identity in $H/H^\circ$ admits an open subgroup $\ti O$ (see \cite[Theorem 2.5]{MZ}). Let $O$ be the pre-image of $\ti O$ in $H$. By Lemma \ref{lem:open} we have a finite group $F$ and an $\gep$-quasi morphism $\rho:F\to O$ with an $\gep$-dense image. For every $\gc\in F$ pick arbitrarily an element in $H^\circ\cap B(\rho(\gc),\gep)$ and denote it by $\rho'(\gc)$. It is easy to varify that, since the metric is bi-invariant, $\rho':F\to H^\circ$ is a $4\gep$-quasi morphism and its image is $2\gep$-dense.
\end{proof}

\begin{prop}[A generalised version of Turing's theorem]\label{prop:turing}
A compact group $H$ is quasi finite iff its identity connected component $H^\circ$ is abelian.
\end{prop}

We will make use of:

\begin{lem}\label{lem:dense-direct}
Let $L$ be a compact Lie group with a commutative identity connected component $L^\circ$. Then $L$ admits a dense subgroup which is a direct limit of finite subgroups.
\end{lem}

In particular, Lemma \ref{lem:dense-direct} implies that $L$ admits a finite subgroup which meets every connected component of $L$. This is a special case of:

\begin{lem}\label{lem:platonov}
Every compact Lie group $K$ admits a finite subgroup $\Lambda$ such that $K=\Lambda K^\circ$, where $K^\circ$ is the identity connected component of $K$.
\end{lem}

\begin{proof}
Recall that there is a complex algebraic group $\BG(\BC)$ with maximal compact subgroup isomorphic to $K$ (see \cite[3.4.4 and 5.2.5]{OV}). By a lemma of Platonov \cite[10.10]{We} (see also \cite{BS}), there is a finite group $\Lambda\le \BG(\BC)$ which meets any connected component of $\BG(\BC)$ (recall that the Zariski connected components of a complex algebraic group coincide with the Hausdorff connected components (see \cite[Theorem 3.5]{PR}). Since $\BG(\BC)$ has finitely many connected components, $\Lambda$, being compact, is conjugated to a subgroup of the maximal compact $K$ (see \cite{M}).
\end{proof}

\begin{proof}[Proof of Lemma \ref{lem:dense-direct}]
Obviously, the identity connected component $L^\circ$, which is a compact torus, admits a dense subgroup $\gC$ which is a direct union of finite groups $\gC=\cup A_n$.
Let $\Lambda$ be a finite subgroup that meets every connected component of $L$ (see Lemma \ref{lem:platonov}).
For every $n$ set $B_n=\langle \Lambda,A_n\rangle$. Then $B_n\cap L^\circ$ is generated by the finite set $\bigcup_{\gl\in\Lambda}A_n^{\gl}\cup(\Lambda\cap L^\circ)$ which consists of torsion elements. Since $B_n\cap L^\circ$ is abelian, it is finite, and therefore $B_n$ is finite as well.
Finally, the direct limit of the $B_n$ is the dense subgroup $\langle\Lambda,\gC\rangle$ of $L$.
\end{proof}

\begin{proof}[Proof of Proposition \ref{prop:turing}]
Let $H$ be a compact group with a bi-invariant metric. By the Peter--Weyl theorem (see \cite[Theorem 3.3.4]{Zimmer}) $H^\circ$ is an inverse limit $\varprojlim H_n$ where the $H_n$ are compact connected Lie groups. If $H$ is quasi finite then, by Lemma \ref{lem:H^0} so is $H^\circ$ and hence also $H_n$ for every $n$. By Turing's theorem \cite{Turing} a connected Lie group which is quasi finite is abelian. Hence the $H_n$ are abelian which in turn implies that $H^\circ$ is abelian. 

Conversely, suppose that $H^\circ$ is commutative and let $\gep>0$. By \cite[Theorem 4.6]{MZ} the $\gep$-neighborhood of identity in $H$ admits a normal subgroup $N$ such that $H/N$ is a Lie group and the identity component of $H$ is mapped onto the identity component of $H/N$. Thus $(H/N)$ is a compact Lie group with a commutative identity component. By Lemma \ref{lem:dense-direct} $H/N$ admits an $\gep$-dense finite subgroup $\gD$. Finally an arbitrary lift $\gD\to H$ is a $4\gep$-quasi morphism with $2\gep$-dense image.
\end{proof}


\section{Completing the proof of the main results}

Let us return now to the proof of Theorem \ref{thm}. Suppose that $X$ is approximable by finite homogeneous metric spaces.
It follows from Proposition \ref{prop:G} and Lemma \ref{lem:8epsilon} that $\text{Isom}(X)$ admits a closed subgroup $G$ with the following properties:
\begin{enumerate}
\item $G$ acts transitively on $X$. 
\item $G$ is quasi finite.
\end{enumerate}
 By Proposition \ref{prop:turing}, $(2)$ implies that the identity connected component $G^\circ$ is abelian. 
Thus we obtained the ``only if" side of Theorem \ref{thm}.

\medskip

In order to prove the "if" side, consider a compact metric space $X$ which admits a compact transitive group of isometries $G$ with $G^\circ$ abelian. 
By \cite[Theorem 4.6]{MZ} there is a descending chain of normal subgroups $K_n\lhd G$ such that $G=\varprojlim G/K_n$ and $G_n=G/K_n$ is a Lie group with a commutative identity connected component. By Lemma \ref{lem:dense-direct}, $G_n$ admits a dense subgroup which is a direct limit of finite groups. In particular every homogeneous $G_n$-space can be approximated by finite homogeneous metric spaces. Let $X_n=K_n\backslash X$ be the orbit space of $K_n$ in $X$ with the induced metric. Then $X=\varprojlim X_n$. It follows that $X_n\to X$ in the Gromov--Hausdorff metric, and since every $X_n$ is approximable by finite homogeneous metric spaces, so is $X$.

\qed

The deduction of Corollary \ref{cor} from Theorem \ref{thm} makes use of:

\begin{lem}\label{lem:connected-transitive}
Let $H$ be a compact group acting continuously and transitively by isometries on a metric space $Y$. Then the identity connected component $H^\circ$ acts transitively on every connected component of $Y$.
\end{lem} 
 
In particular, the lemma says that a homogenuous space of a pro-finite group is totally disconnected, and indeed, if $P$ is a pro-finite group and $L\le P$ is a closed subgroup, the sets $KhL$ where $K$ runs over the open subgroups of $P$ and $h\in P$ is arbitrary, form a base for the topology of $P/L$ consisting of open-closed sets.
 
\begin{proof}
 Let $Y^\circ$ be a connected component of $Y$ and let $H_1$ be the stabilizer of $Y^\circ$, $H_1=\{g\in H: g\cdot Y^\circ=Y^\circ\}$. Then $H_1$ is a group containing $H^\circ$. Let $P=H_1/H^\circ$ and let $Y^\circ/\sim$ be the orbit space of $H^\circ$ in $Y^\circ$ with the quotient metric. Then $P$ is a pro-finite group acting transitively on the connected space $Y^\circ/\sim$. It follows that $Y^\circ/\sim$ is a singleton, i.e. that $H^\circ$ is transitive on $Y^\circ$.
\end{proof}
 
\begin{proof}[Proof of Corollary \ref{cor}]
Suppose that $X$ is a compact metric space which is approximable by finite homogeneous ones. By Theorem \ref{thm}, $\text{Isom}(X)$ has a transitive compact subgroup $G$ with $G^\circ$ abelian. Let $X^\circ$ be a connected component of $X$, $x_0\in X^\circ$ a point and $Q=(G^\circ)_{x_0}$ its stabilizer group in $G^\circ$. By Lemma \ref{lem:connected-transitive}, $G^\circ$ acts transitively on $X^\circ$ and hence $X^\circ$ is homeomorphic to $G^\circ/Q$. 
Since $G^\circ$ is abelian, $Q$ is normal and $G^\circ/Q$ is a compact abelian group. By the Peter--Weyl theorem $G^\circ/Q=\varprojlim T_n$ where each $T_n$ is a finite dimensional torus. 
\end{proof}

\medskip

We end this note with a:

\begin{proof}[Proof of Corollary \ref{cor:distance-transitive}]

Let $(X_n)$ be a sequence as in \ref{cor:distance-transitive} and suppose that $X_n\to X$. Then $X$ is a length space and hence, by Theorem \ref{thm} (see Remark \ref{rem}.4), a torus (of finite or infinite dimension). Let $x\in X$ be a point, and for $r>0$ consider the $r$-sphere $S_X(r,x)$ around $x$ in $X$. It is easy to see that $S_X(r,x)$ is the limit of $S_{X_n}(r_n,x_n)$ for an appropriate choice of $r_n\to r$ and $x_n\in X_n$. 
If $\dim X>1$ then there is $r$ and a connected component of $S_X(r,x)$ which is not a singleton, and we claim that this is impossible. Indeed, letting $G$ denote a limit group of $G_n=\text{Isom}(X_n)$, it follows from Proposition \ref{prop:G} and Lemma \ref{lem:connected-transitive} that the identity connected component $G^\circ$ acts transitively on $X$. 
Similarly, since the stabilizer group $G_x$ contains a limit of the the stabilizer group $(G_n)_{x_n}$, and the later act transitively on the corresponding $S_{X_n}(r_n,x_n)$,
we deduce that its identity connected component $(G_x)^\circ$ acts transitively on every connected component of $S_X(r,x)$.
However, if $b\in G^\circ$ is an element that takes $x$ into a non-singleton connected component of $S_X(r,x)$, and $a\in G_x^\circ\le G^\circ$ is an element that moves $b\cdot x$, it is easy to see that $a$ and $b$ do not commute, in contrary to the commutativity of $G^\circ$ which is guaranteed by Theorem \ref{thm}. 
\end{proof}


\begin{thebibliography}{999999}

\bibitem[AGG01]{AGG} M. A. Alekseev, L. Yu. Glebskii, and E. I. Gordon, On approximation of groups, group actions and Hopf algebras, Journal of Mathematical Sciences, Vol. 107, No. 5, 2001.

\bibitem[BFT12]{Itai} I. Benjamini, H. Finucane, R. Tessera, On the scaling limit of finite vertex transitive graphs with large diameter, arXiv:1203.5624.

\bibitem[BS64]{BS} A. Borel, J.P. Serre, 
ThŽormes de finitude en cohomologie galoisienne,
Comment. Math. Helv. 39 1964 111--164.

\bibitem[BH99]{BH} M. Bridson, A. Haefliger, \emph{Metric Spaces of Non-Positive
Curvature}, Grundl. der Math. Wiss. 319, Springer Verlag, 1999.

\bibitem[J1878]{J} C.  Jordan, Mkmoire  sur  les \'equations  diffir\'entielles  liniaires  \'a  int\'egrale  alg\'ebrique,  J. 
Math.  84  (1878),  89--215.

\bibitem[K82]{Kazhdan} D. Kazhdan, On $\gep$-representations. Israel J. Math. 43 (1982), no. 4, 315--323.

\bibitem[MZ55]{MZ} D. Montgomery, L. Zippin, Topological transformation groups. Reprint of the 1955 original. Robert E. Krieger Publishing Co., Huntington, N.Y., 1974. 

\bibitem[M55]{M} G.D. Mostow, Self adjoint groups. Ann. of Math. 62 (1955) pp.
44--55.

\bibitem[OV90]{OV} A.L. Onishchik, E.B.  Vinberg, Lie groups and algebraic groups. Translated from the Russian and with a preface by D. A. Leites. Springer Series in Soviet Mathematics. Springer-Verlag, Berlin, 1990. 

\bibitem[PR91]{PR} V. Platonov, A. Rapinchuk, Algebraic groups and number theory. Translated from the 1991 Russian original by Rachel Rowen. Pure and Applied Mathematics, 139. Academic Press, Inc., Boston, MA, 1994. 
  
\bibitem[T38]{Turing} A.M. Turing, Finite approximations to Lie groups. 
Ann. of Math. (2) 39 (1938), no. 1, 10--111.

\bibitem[W73]{We} B.A.F Wehrfritz, Infinite linear groups. An account of the group-theoretic properties of infinite groups of matrices. Ergebnisse der Matematik und ihrer Grenzgebiete, Band 76. Springer-Verlag, New York-Heidelberg, 1973. 

\bibitem[Z90]{Zimmer} R.J. Zimmer, Essential Results of Functional Analysis, University of Chicago Press, Chicago, IL, 1990. 
\end{thebibliography}
\end{document}